\documentclass[a4paper,12pt]{amsart}
\usepackage[utf8x]{inputenc}
\usepackage{amsmath,amssymb,amsthm}
\usepackage[all]{xy}
\usepackage{color}
\usepackage{hyperref}

\newtheorem{theorem}{Theorem}[section]

\newtheorem{lemma}[theorem]{Lemma}
\newtheorem{proposition}[theorem]{Proposition}

\theoremstyle{definition}

\theoremstyle{remark}
\newtheorem*{remark}{Remark}
\newcommand{\disc}{\mathop{\rm disc}}

\newcommand{\Gal}{\mathop{\textnormal{Gal}}}
\renewcommand{\phi}{\varphi}

\begin{document}

\title[Hardy-Littlewood conjecture over large finite fields]{Hardy-Littlewood tuple conjecture over large finite fields}

\author{Lior Bary-Soroker}

\begin{abstract}
We prove the following function field analog of the Hardy-Littlewood conjecture (which generalizes the twin prime conjecture) over large finite fields. 
Let $n,r$ be positive integers and $q$ an odd prime power. For a tuple of  distinct polynomials $\mathbf{a}=(a_1, \ldots, a_r)\in \mathbb{F}_q[t]^r$ of degree $<n$ let $\pi(q,n;\mathbf{a})$ be the number of monic polynomials $f\in \mathbb{F}_q[t]$ of degree $n$ such that $f+a_1, \ldots, f+a_r$ are simultaneously irreducible. We prove that $\pi(q,n;\mathbf{a})\sim\frac{q^n}{n^r}$ as $q\to \infty$, $q$ odd, and $n,r$ fixed.
\end{abstract}

\maketitle

\section{Introduction}
The twin prime conjecture predicts that there are infinitely many positive integers $n$ such that $n$ and $n+2$ are primes. In other words, if 
\[
\pi_2(x) = \#\{1\leq  n\leq x \mid n \mbox{ and } n+2 \mbox{ are primes}\}
\]
be the corresponding counting function, then the conjecture says that $\pi_2(x) \to \infty$ as $x\to \infty$. A more precise conjecture, the Hardy-Littlewood conjecture, predicts that
\[
\pi_2(x)\sim2\prod_{p>2}\frac{p(p-2)}{(p-1)^2} \frac{x}{\log^2 x}\approx 1.32 \frac{x}{\log^2 x}, \quad x\to \infty.
\]
More generally, let $\mathbf{a} = (a_1, \ldots, a_r)\in \mathbb{Z}^r$ be an $r$-tuple of distinct integers and let 
\[
\pi(x;\mathbf{a})= \#\{ 1\leq n\leq x \mid n+a_1, \ldots, n+a_r \mbox{ are all prime}\}
\]
is the corresponding prime counting function. 
Then $\pi(x) = \pi(x;0)$ and $\pi_2(x) = \pi(x;0,2)$. If $a_1, \ldots, a_r$ cover all residues modulo some prime $p$, e.g.\ $(a_1,a_2,a_3)=(0,2,4)$ and $p=3$, then for any $n$ there is $i$ such that $p\mid n+a_i$. In particular, $\pi(x;\mathbf{a})$ is bounded as $x\to \infty$. Otherwise the Hardy-Littlewood conjecture \cite{HardyLittlewood} says that  $\pi(x;\mathbf{a}) \to \infty$ as $x\to \infty$.

As in the twin prime conjecture, the Hardy-Littlewood conjecture gives the rate in which $\pi(x;\mathbf{a})$ tends to infinity: Let
\[
\nu(p; \mathbf{a}) = \#\{ a_1\mod p, \ldots, a_r \mod p\}
\]
and 
\[
C(\mathbf{a})=\prod_{p} \frac{1-\nu(p;\mathbf{a})/p}{(1-1/p)^r}.
\] 
It is an exercise in analytic number theory that, unless $a_1, \ldots, a_r$ cover all residues modulo some prime $p$,  the product converges, i.e.\ $C(\mathbf{a})>0$. Then the Hardy-Littlewood conjecture predicts that
\[
\pi(x;\mathbf{a}) \sim C(\mathbf{a})\frac{x}{\log^r x}, \quad x\to \infty.
\]

The objective of this paper is to prove an analog of the Hardy-Littlewood conjecture over \emph{large} finite fields of odd cardinality. Let $q$ be a prime power, let $\mathbb{F}_q$ be the finite field of $q$ elements, and let $n$ and $r$ be integers. For an $r$-tuple of distinct polynomials $\mathbf{a} = (a_1, \ldots, a_r)\in \mathbb{F}_q[t]^r$, each of degree  $<n$, we let 
\[
\begin{array}{ll}
\pi(q,n;\mathbf{a}) =\#\{ f\in \mathbb{F}_q[t] \mid &\mbox{$f$ is monic and of degree $n$ and}\\
&\mbox{$f+a_1, \ldots, f+a_r$ are all irreducible}\}.
\end{array}
\]
Since $q^n$ is the number of degree $n$ monic polynomials, it  plays the role of $x$ in this setting. Therefore it is desirable to estimate $\pi(q,n;\mathbf{a})$ as $q^n\to \infty$. The straightforward analog of the classical setting, and the more difficult case, is when $q$ is fixed and $n\to \infty$. In this paper we treat the asymptotic when $n$ is fixed and  $q\to \infty$:

\begin{theorem}\label{thm:main}
Let $n$ and $r$ be positive integers and $q$ an odd prime power. Then for every distinct $a_1, \ldots, a_r\in \mathbb{F}_q[t]$, each of degree $<n$, we have 
\[
\pi(q,n;\mathbf{a}) = \frac{q^n}{n^r} + O_{n,r}(q^{n-\frac12}).
\]
\end{theorem}

\begin{remark}
In contrast to the original Hardy-Littlewood conjecture, in Theorem~\ref{thm:main} the $a_1, \ldots, a_r$ do not have to be fixed in advance. Actually even the characteristic of the ring $\mathbb{F}_q$ may vary. 
\end{remark}

\begin{remark}
Our proof may give an explicit bound of the implied constant in the above theorem. However this bound is not so good, as its dependence on $n$ is worse than $n!$. For this reason, in this paper, we avoid tracking  this bound. 
\end{remark}

\begin{remark}
In \cite{BenderPollack} Bender and Pollack prove Theorem~\ref{thm:main} for $r=2$. 
\end{remark}

\begin{remark}
Theorem~\ref{thm:main} extends a special case of a result of Pollack \cite{Pollack2008} (provided $\gcd(2n,q)=1$) and the author \cite{BarySoroker2012} (provided either $q$ odd or $q$ even and $n$ odd) from constant $a_1, \ldots, a_r$ to polynomials. 
\end{remark}
\subsection*{Outline of the proof}
Our approach is generic in the following sense: We take $\mathcal{F}$ to be a degree $n$ monic polynomial with variable coefficients and count in how many ways we can specialize the coefficients of $\mathcal{F}$ to $\mathbb{F}_q$ such that $f+a_1, \ldots, f+a_r$ are irreducible, where $f$ is the specialized polynomial. 

In order to achieve this we use an irreducibility criterion, whose proof is based on an earlier result of the author reducing the problem to a problem on rational points. Then the proof applies the Lang-Weil estimates (Section~\ref{sec:irr}). 

In order to apply the irreducibility criterion we need to calculate the Galois group of $\prod_{i=1}^r (\mathcal{F}+a_i)$. Using a group theoretical lemma, this calculation is reduced to proving square independence of discriminants which is achieved by applying a result of Carmon and Rudnick  \cite{CarmonRudnick} (Section~\ref{sec:Galois}).

\section{Irreducibility Result}\label{sec:irr}
We denote by $\tilde{\mathbb{F}}_q$ a fixed algebraic closure of $\mathbb{F}_q$.

\begin{proposition}\label{prop:chobotarev}
Let $(\mathbf{U},t)=(U_1, \ldots, U_n,t)$ be an $(n+1)$-tuple of variables, let $\mathcal{F}_1,\ldots, \mathcal{F}_r\in \mathbb{F}_q[\mathbf{U},t]$ be polynomials each of degree $n$ in $t$ and with respective splitting fields $F'_1, \ldots, F'_r$ over $E'=\mathbb{F}_q(\mathbf{U})$, and let $F' = F'_1\cdots F'_r$. Let $E=E'\cdot \tilde{\mathbb{F}}_q$, and similarly $F_i = F'_i\cdot \tilde{\mathbb{F}}_q$ and $F=F'\cdot\tilde{\mathbb{F}}_q$.  Assume that 
\[
\Gal(F/E)\cong  S_n^r.
\]
Then the number of $\mathbf{u}=(u_1, \ldots,u_n)\in \mathbb{F}_q^n$ for which all the specialized polynomials $f_i(t)=\mathcal{F}_i(\mathbf{u},t)$ are irreducible is $q^n/n^r + O_{n,r}(q^{n-1/2})$ as $q\to \infty$ and $n,r$ are fixed.
\end{proposition}

To prove this we apply \cite[Lemma 2.8]{BarySorokerJarden2012} (which is a special case of \cite[Lemma 2.2]{BarySoroker2012}), which we quote here for the reader's convenience:

\begin{lemma}\label{lem:tech}
Let $K$ be a field, let $(\mathbf{U},t) = (U_1,\ldots,U_n,t)$ an $(n+1)$-tuple of variables over $K$.
For each $1\le i\le r$ let $\mathcal{F}_i\in K[\mathbf{U},t]$ be an absolutely irreducible polynomial 
which is separable and of degree $n_i$ in $t$.
Let $L_i$ be a Galois
extension of $K$ of degree $n_i$.

We denote the splitting field of $\mathcal{F}=\mathcal{F}_1\cdots \mathcal{F}_r$,
considered as a polynomial in $t$, over $E'=K(\mathbf{U})$ by $F'$
and assume that $F'/K$ is regular and $\Gal(F'/E')\cong\prod_{i=1}^r S_{n_i}$.

Then there exist 
a proper algebraic subset $Z$ of $\mathbb{A}_K^n$,
an absolutely irreducible normal affine
$K$-variety $W$, and a finite \'etale map $\rho\colon W\to V$
with $V=\mathbb{A}_K^n\smallsetminus Z$
such that the following conditions hold:
\begin{enumerate}
\item
$Z$ and $W$ are defined in $\mathbb{A}_K^n$ and $\mathbb{A}_K^{n+3}$,
respectively, by polynomials with coefficients in
$K$ whose degrees are bounded in terms of $n,r$.
\item
$\deg(\mathcal{F}_i(\mathbf{u},t))=n_i$ for each $\mathbf{u}\in V(K)$ and  $i=1,\ldots,r$.
\item
$\rho(W(K))$ is the set of all $\mathbf{u}\in V(K)$ such that
$L_i$ is generated by a root of $\mathcal{F}_i(\mathbf{u},t)$,
$i=1,\ldots,r$.
In particular, $\mathcal{F}_i(\mathbf{u},t)$ is irreducible of degree $n_i$ for
all $\mathbf{u}\in\rho(W(K))$ and $1\le i\le r$.
\item
$|(\rho')^{-1}(\mathbf{u})\cap W(K)|=\prod_{i=1}^rn_i$
for all $\mathbf{u}\in\rho(W(K))$.
\end{enumerate}
\end{lemma}

\begin{remark}
In \cite[Lemma 2.8]{BarySorokerJarden2012} the condition that $F'$ is regular over $K$ was mistakenly omitted. See erratum of \cite{BarySorokerJarden2012} for further details. 
\end{remark}

\begin{proof}[Proof of Proposition~\ref{prop:chobotarev}]
Let $K=\mathbb{F}_q$. Since 
\[
S_n^r\cong \Gal(F/E)\hookrightarrow \Gal(F'/E')\hookrightarrow S_n^r,
\]
we get that $\Gal(F/E)\cong \Gal(F'/E')\cong S_n^r$ and thus that $F'/K$ is regular. In particular, each $\mathcal{F}_i$ is absolutely irreducible and separable. For each $i$, we let $L_i=\mathbb{F}_{q^n}$ be the unique extension of $\mathbb{F}_q$ of degree $n$. 

By Lemma~\ref{lem:tech} (with $n_1=\cdots = n_r = n$) there exist a proper algebraic subset $Z$ of $\mathbb{A}_K^n$,
an absolutely irreducible normal affine
$K$-variety $W'$, and a finite \'etale map $\rho'\colon W'\to V$
with $Z=\mathbb{A}_K^n\smallsetminus V$ that satisfy conditions (1)-(4) of the lemma. 

Since $Z$ is a closed subset of $\mathbb{A}_K^n$ defined by polynomials of bounded degrees in terms of $n,r$ we get that $\# Z(K)=O_{n,r}(q^{n-1})$ (see e.g.\ \cite[Lemma~1]{LangWeil54}). Note that $\dim W'=n$, so the Lang-Weil estimates give 
\[
\#W'(K)=q^n +  O_{n,r}(q^{n-\frac{1}{2}}).
\]
(See \cite[Theorem 1]{LangWeil54}. In \cite[Theorem 2.1]{Zywina2010} Zywina gives the best known bounds on the implied constant.)

Let $S$ be the set of all $\mathbf{u}\in K^n$ such that $\mathcal{F}_i(\mathbf{u},t)$ is irreducible in $K[t]$, for every $i=1,\ldots, r$. By Lemma~\ref{lem:tech}(3), $\rho'(W'(K))$ is the set of all $\mathbf{u}\in V(K)\subseteq K^n$ such that $\mathbb{F}_{q^n}$ is generated by a root of $\mathcal{F}_i(\mathbf{u},t)$ over $K$, for every $i=1,\ldots, r$. Since $\mathcal{F}_i(\mathbf{u},t)$ is of degree $n$, for $\mathbf{u}\in V(K)$, and $\mathbb{F}_{q^n}$ is the unique extension of degree $n$ of $K$, we get that $\rho'(W'(K)) = S\cap V(K)$, hence
$\#\rho'(W'(K))  =  \#(S\cap V(K))$.

Lemma~\ref{lem:tech}(4) gives that $\#((\rho')^{-1}(\mathbf{u}) \cap W'(K)) = n^r$, for every $\mathbf{u} \in \rho'(W'(K)) $. We therefore get that 
\begin{eqnarray*}
\#S &=& \#(S\cap V(K)) + \#(S\cap Z(K))  =\#\rho'(W'(K)) + O_{n,r}(q^{n-1})\\ 
&=& \frac{1}{n^r} \#W'(K) + O_{n,r}(q^{n-1}) = \frac{ q^n }{n^r}+ O_{n,r}(q^{n-\frac{1}{2}}),
\end{eqnarray*}
as needed. 
\end{proof}

\section{Galois groups}\label{sec:Galois}
Let $q$ be an odd prime power, let $n,r$ be positive integers, let $\mathbf{a} = (a_1, \ldots, a_r)\in \mathbb{F}_q[t]^r$ be an $r$-tuple of distinct polynomials each of degree $<n$, let $\mathbf{U}=(U_1, \ldots, U_n)$  be an $n$-tuple of variables over $\mathbb{F}_q$, and let $\mathcal{F}=t^n + U_1 t^{n-1} + \cdots + U_n\in \mathbb{F}_q[\mathbf{U},t]$. For each $i=1, \ldots, r$, let $\mathcal{F}_i = \mathcal{F}+a_i$. We let $E=\tilde{\mathbb{F}}_q(\mathbf{U})$, for each $i=1, \ldots, r$, we let $F_i$ be the splitting field of $\mathcal{F}_i$ over $E$, and we denote by $F$ the compositum of $F_1, \ldots, F_r$. Thus $F$ is the splitting field of $\prod_{i=1}^r \mathcal{F}_i$ over $E$. 
Since each of the $\mathcal{F}_i$ has variable coefficients, we have an isomorphism 
\begin{equation}\label{eq:Sn}
\Gal(F_i/E)\cong S_n
\end{equation}
that is induced form the action of $\Gal(F_i/E)$ on the roots of $\mathcal{F}_i$ in $F_i$ (here $S_n$ is the symmetric group of degree $n$). Thus the restriction maps $r_i \colon \Gal(F/E) \to \Gal(F_i/E)$, $i=1, \ldots, r$, induce an embedding $\phi\colon  \Gal(F/E)\to\prod_{i=1}^r \Gal(F_i/E)\cong S_n^r$. 

\[
\xymatrix{
&F\\
F_1\ar@{-}[ur]\ar@{.}[rr]
				&& F_r\ar@{-}[ul]\\
&E=\tilde{\mathbb{F}}_q(\mathbf{U})\ar@{-}[ur]^{S_n}\ar@{-}[ul]_{S_n}
}
\]

\begin{proposition}\label{prop:Galois}
Under the notation above, 
$\phi \colon \Gal(F/E) \to S_n^r$ is an isomorphism.
\end{proposition}

The proof the Proposition~\ref{prop:Galois} appears after the following two lemmas. 

For a separable polynomial $f\in E[t]$ we denote by $\delta_t(f)$ the square class of its discriminant $\disc_t(f)$ in the $\mathbb{F}_2$-vector space $E^\times/(E^\times)^2$. (Recall that the characteristic of $E$ is not $2$.)

\begin{lemma}[Carmon-Rudnick]\label{lem:CR}
The number $N$ of $\mathbf{u} = (u_1, \ldots, u_n-1)\in \mathbb{F}_q^{n-2}$ for which  $\disc_t(\mathcal{F}(\mathbf{u},U_n,t)+a_1), \ldots  , \disc_t(\mathcal{F}(\mathbf{u},U_n,t)+a_r)$ are square-free, coprime, and non-constant is $q^{n-1} + O_{n,r}(q^{n-2})$.
\end{lemma}

\begin{proof}
In  \cite[Page 3]{CarmonRudnick}, for some $i$, a set $G_n$ is defined to be the tuples $\mathbf{u}$ such that $\disc_t(\mathcal{F}(\mathbf{u},U_n,t)+a_i)$ is square-free of positive degree, and for every $j\neq i$, $\disc_t(\mathcal{F}(\mathbf{u},U_n,t)+a_j)$ and $\disc_t(\mathcal{F}(\mathbf{u},U_n,t)+a_i)$ are coprime. They show $\# G_n = q^{n-1}+O_{n,r}(q^{n-2})$. 

Let us denote the above set by $G_{n,i}$ to keep track of the dependence of $i$. Then $N = \# (\bigcap_{i=1}^n G_{n,i}) = q^{n-1} + O_{n,r} (q^{n-2})$, as claimed.   
\end{proof}

\begin{lemma}\label{lem:ind}
The square classes $\delta_t(\mathcal{F}_1), \ldots, \delta_t(\mathcal{F}_r)$ are linearly independent in $E^\times/(E^\times)^2$.
\end{lemma}

\begin{proof}
For a specialization $\epsilon\colon (U_1, \ldots, U_{n-1}) \mapsto (u_1, \ldots, u_{n-1})\in \mathbb{F}_{q^e}^{n-1}$ we let $f_\epsilon= t^n + u_1 t^{n-1} + \cdots + u_{n-1}t + U_n\in \mathbb{F}_{q^e}[U_n,t]$ be the specialized polynomial. By Lemma~\ref{lem:CR} for $q^{e(n-1)} + O_{n,r}(q^{e(n-2)})$ such specializations $\epsilon$ the discriminants $\disc_t(f_\epsilon+a_1), \ldots  , \disc_t(f_\epsilon+a_r)$ are square free, coprime, and non-constant. In particular if $e$ is sufficiently large there exists at least one such specialization in $\mathbb{F}_{q^e}\subseteq E$. Thus the square classes $\delta_t(f_\epsilon+a_i)$ of $\disc_t(f_\epsilon+a_i)$ in the $\mathbb{F}_2$-vector space $\mathbb{F}_q(U_n)^\times/(\mathbb{F}_q(U_n)^\times)^2$ are linearly independent. Since squares cannot be specialized to non-squares, we get that $\delta_t(\mathcal{F}_1), \ldots, \delta_t(\mathcal{F}_r)$ are linearly independent. 
\end{proof}

Another ingredient in the proof of Proposition~\ref{prop:Galois} is the following linearly disjointness criterion that is proved in \cite[Lemma 3.4]{BarySoroker2012}:

\begin{lemma}\label{lem:discimpliesdisj}
If $\Gal(\mathcal{F}_i,E)=\Gal(F_i/E)=S_n$ for all $i=1,\ldots, r$ and if $\delta_t(\mathcal{F}_1), \ldots, \delta_t(\mathcal{F}_r)$ are linearly independent in $E^\times/(E^\times)^2$, then $F_1,\cdots, F_r$ are linearly disjoint over $E$. 
\end{lemma}

\begin{proof}[Proof of Proposition~\ref{prop:Galois}]
Lemma~\ref{lem:ind} and \eqref{eq:Sn} imply that the conditions of Lemma~\ref{lem:discimpliesdisj} are met. Therefore $F_1,\ldots, F_r$ are linearly disjoint, hence $\phi$ is an isomorphism (\cite[VI.1.15, \S1]{Lang}). 
\end{proof}

\section{Proof of Theorem~\ref{thm:main}}
Let $n$ and $r$ be fixed positive integers and let $q$ be an odd prime power. Let $a_1, \ldots, a_r\in \mathbb{F}_q[t]$ be distinct polynomials, each of degree $<n$. Let $\mathbf{U}=(U_1, \ldots, U_n)$ be an $n$-tuple of variables over $\mathbb{F}_q$, and let $\mathcal{F}=t^n + U_1 t^{n-1} + \cdots + U_n\in \mathbb{F}_q[\mathbf{U},t]$. For each $i=1, \ldots, r$, let $\mathcal{F}_i = \mathcal{F}+a_i$. Proposition~\ref{prop:Galois} shows that the assumptions of Proposition~\ref{prop:chobotarev} are met. 

Thus the number of $\mathbf{u} \in \mathbb{F}_q$ for which $\mathcal{F}(\mathbf{u},t)+a_i = \mathcal{F}_i(\mathbf{u},t)$ are simultaneously irreducible in $\mathbb{F}_q[t]$ is $\frac{q^{n}}{n^r} + O_{n,r}(q^{n-\frac12})$. This finishes the proof since this number equals $\pi(q,n;\mathbf{a})$. 
\qed

\section*{Acknowledgments}
The author wishes to thank Ze\'ev Rudnick for raising the question that led to this research and to Benjamin Collas and Arno Fehm for many helpful remarks that improved the presentation of this paper.  
\bibliographystyle{plain}
\bibliography{bib}

\end{document}